\documentclass[12pt]{article}
\usepackage{amsmath}
\usepackage{amssymb}
\usepackage{amsmath,amssymb,amsbsy,amsfonts,amsthm,latexsym,
amsopn,amstext,amsxtra,euscript,amscd}

\newtheorem{lem}{Lemma}
\newtheorem{lemma}[lem]{Lemma}

\newtheorem{thm}{Theorem}
\newtheorem{theorem}[thm]{Theorem}
\newtheorem{cor}{Corollary}
\newtheorem{corollary}[cor]{Corollary}

\begin{document}

\def\F{{\mathbb F}}
\def\Z{{\mathbb Z}}
\def\N{{\mathbb N}}

\def\cA{{\mathcal A}}
\def\cB{{\mathcal B}}
\def\cC{{\mathcal C}}
\def\cD{{\mathcal D}}
\def\cE{{\mathcal E}}
\def\cG{{\mathcal G}}
\def\cI{{\mathcal I}}
\def\cJ{{\mathcal J}}
\def\cQ{{\mathcal Q}}
\def\cR{{\mathcal R}}
\def\cT{{\mathcal T}}
\def\cX{{\mathcal X}}
\def\cY{{\mathcal Y}}
\def\cP{{\mathcal P}}
\def\cH{{\mathcal H}}
\def\cU{{\mathcal U}}
\def\cV{{\mathcal V}}
\def\cS{{\mathcal S}}

\def\ep{\mathbf{e}_p}

\def\ek{\mathbf{e}_k}

\def\em{\mathbf{e}_m}

\def\ems{\mathbf{e}_{m/s}}

\baselineskip 15pt

\title{On congruences involving product of variables from short intervals}

\author{
{\sc M. Z. Garaev} }

\date{\empty}

\pagenumbering{arabic}

\maketitle


\begin{abstract}
We prove several results which imply the following consequences.

For any $\varepsilon>0$ and any sufficiently large prime $p$,  if $\cI_1,\ldots, \cI_{13}$ are intervals of cardinalities $|\cI_j|>p^{1/4+\varepsilon}$ and $abc\not\equiv 0\pmod p$, then
the congruence
$$
ax_1\cdots x_6+bx_7\cdots x_{13}\equiv c\pmod p
$$
has a solution with $x_j\in\cI_j$.

There exists an absolute constant $n_0\in\N$ such that for any $0<\varepsilon<1$ and any sufficiently large prime $p$, any quadratic residue $\lambda$ modulo $p$ can be represented in the form
$$
x_1\cdots x_{n_0}\equiv \lambda\pmod p,\quad x_i\in\N,\quad x_i\le p^{1/(4e^{2/3})+\varepsilon}.
$$

For any $\varepsilon>0$ there exists $n=n(\varepsilon)\in \N$ such that for any sufficiently large $m\in\N$ the congruence
$$
x_1\cdots x_{n}\equiv 1\pmod m,\quad x_i\in\N,\quad x_i\le m^{\varepsilon}
$$
has a solution with $x_1\not=1$.

\end{abstract}



\section{Introduction}

For a prime $p$, let  $\F_p$ denote the field of residue classes modulo $p$ and $\F_p^*$ be the set of nonzero elements of $\F_p$.

Let $\cI_1,\ldots,\cI_{2k}$ be nonzero intervals in $\F_p$ and  let $\cB$ be the box
$$
\cB=\cI_1\times \cI_2\times\ldots \times \cI_{2k}.
$$
We recall that  a set $\cI\subset\F_p$ is called an interval if
$$
\cI=\{L+1,\ldots, L+N\}\!\!\!\pmod p
$$
for some integers $L$ and $N\ge 1$.

Given elements $a,b\in\F_p^*$ and $c\in\F_p$, we consider the equation
\begin{equation}
\label{eqn:main1}
ax_1\cdots x_k+bx_{k+1}\cdots x_{2k} =c;\quad  (x_1,\ldots,x_{2k})\in\cB.
\end{equation}
The problem is to determine how large the size of the  box $\cB$ should be in order to guarantee the solvability of~\eqref{eqn:main1}.

The case $k=2$ was initiated in the work of Ayyad, Cochrane and Zhang~\cite{ACZh}, and then continued in~\cite{GGar} and~\cite{AC2}.
It was proved in~\cite{ACZh} that there is a constant $C$ such that if $|\cB|>Cp^2\log^4p$, then the equation
\begin{equation}
\label{eqn:k=2}
ax_1 x_2+bx_3x_4 =c;\quad  (x_1,x_2,x_3,x_4)\in\cB.
\end{equation}
has a solution, and they asked whether the factor $\log^4p$ can be removed. The authors of~\cite{GGar}
relaxed the condition to $|\cB|>Cp^2\log p$ and also proved that~\eqref{eqn:k=2} has a solution in any box $\cB$ with $|\cI_1||\cI_3|>15p$ and $|\cI_2||\cI_4|>15p$.
The main question for $k=2$ was solved by Bourgain (unpublished);
he proved that~\eqref{eqn:k=2} has a solution in any box $\cB$ with $|\cB|\ge Cp^{2},$ for some constant $C$.

The case $k\ge 3$ was a subject of investigation of a recent work of Ayyad and Cochrane~\cite{AC}. They proved a number of results and conjectured
that for fixed $k\ge 3$ and $\varepsilon>0$, if $a,b,c\in \F_p^*$, then
there exists a solution of~\eqref{eqn:main1} in any box $\cB$ with $|\cB|>Cp^{2+\varepsilon}$, for some $C=C(\varepsilon,k)$.

Given two sets $\cA,\cB\subset \F_p$, the sum set $\cA+\cB$ and the product set $\cA\cB$ are defined as
$$
\cA+\cB=\{a+b;\, a\in \cA, \, b\in\cB\},\quad \cA\cB=\{ab;\, a\in \cA, \, b\in\cB\}.
$$
For a given $\xi\in \F_p$ we also use the notation
$$
\xi \cA= \{\xi a;\, a\in \cA\},
$$
so that the solvability of~\eqref{eqn:main1} can be restated in the form
$$
c\in a\prod_{i=1}^k\cI_i + b\prod_{i=k+1}^{2k} \cI_i.
$$

In the present paper we prove the following theorems which improve some results of Ayyad and Cochrane for $k\ge 7$ (see, Table 1 of~\cite{AC}).

\begin{theorem}
\label{thm:13 int}
For any $\varepsilon>0$ there exists $\delta=\delta(\varepsilon)>0$ such that the following holds for any sufficiently large prime $p$: let $\cI_1,\cI_2,\ldots,\cI_{13}\subset \F_p^*$ be intervals with
$$
|\cI_i|>p^{1/4-\delta},\quad i=1,2,\ldots, 12;\quad |\cI_{13}|>p^{1/4+\varepsilon}.
$$
Then for any $a,b,c\in \F_p^*$ we have
$$
c \in a\prod_{i=1}^6\cI_i + b\prod_{i=7}^{13} \cI_i.
$$
\end{theorem}

If we allow $1\in \cI_{13}$, then the condition on the size of $\cI_{13}$ can be relaxed to $|\cI_{13}|>p^{\varepsilon}.$

\begin{theorem}
\label{thm:12 +1 int}
For any $\varepsilon>0$ there exists $\delta=\delta(\varepsilon)>0$ such that the following holds for any sufficiently large prime $p$: let $\cI_1,\cI_2,\ldots,\cI_{13}\subset \F_p^*$ be intervals with
$1\!\!\pmod p\in \cI_{13}$ and
$$
|\cI_i|>p^{1/4-\delta},\quad i=1,2,\ldots, 12;\quad |\cI_{13}|>p^{\varepsilon}.
$$
Then for any $a,b,c\in \F_p^*$ we have
$$
c \in a\prod_{i=1}^6\cI_i + b\prod_{i=7}^{13} \cI_i.
$$
\end{theorem}

From Theorem~\ref{thm:12 +1 int} we have the following consequence.

\begin{corollary}
\label{cor:12 int}
For any $\varepsilon>0$ there exists $\delta=\delta(\varepsilon)>0$ such that the following holds
for any sufficiently large prime $p$:  let $\cI_1,\cI_2,\ldots,\cI_{12}\subset \F_p^*$ be intervals satisfying $1\!\!\pmod p\in \cI_{12}$ and
$$
|\cI_i|>p^{1/4-\delta},\quad i=1,2,\ldots, 11;\quad |\cI_{12}|>p^{1/4+\varepsilon}.
$$
Then for any $a,b,c\in \F_p^*$ we have
$$
c \in a\prod_{i=1}^6\cI_i + b\prod_{i=7}^{12} \cI_i.
$$

\end{corollary}

Indeed, if we set
$$
\cI_{12}'=\{1,2,\ldots, \lfloor p^{1/4+\varepsilon/2}\rfloor\}\!\!\!\pmod p,\quad \cI_{13}=\{1,2,\ldots, \lfloor p^{\varepsilon/2}\rfloor\}\!\!\!\pmod p,
$$
then under the condition of Corollary~\ref{cor:12 int} we have $\cI_{12}\supset \cI'_{12}$ $\cI_{13}$, and the claim follows from the application of Theorem~\ref{thm:12 +1 int}.

We remark that we state and prove our results for intervals of $\F_p^*$ rather than of $\F_p$ just for the sake of simplicity. Indeed, this restriction is not essential, as any nonzero interval $\cI\subset\F_p$ contains an interval $\cI'\subset \cI$ such that $0\not\in\cI'$
and $|\cI'|\ge |\cI|/3$.

\bigskip

In the case $b=0$, the equation~\eqref{eqn:main1} describes the problem of representability of residue classes by product of variables from corresponding intervals.
We shall consider the case when the variables are small positive integers.
It is known from~\cite{Gar} that for any $\varepsilon>0$ and a sufficiently large cube-free  $m\in\N$, every $\lambda$ with
$\gcd(\lambda,m)=1$ can be represented in the form
$$
x_1\cdots x_{8}\equiv \lambda\!\!\!\pmod m, \quad x_i\in\N,\quad x_i\le m^{1/4+\varepsilon}.
$$
Under the same condition, Harman and Shparlinski~\cite{HSh} proved that $\lambda$ can be represented in the form
$$
x_1\cdots x_{14}\equiv \lambda \!\!\!\pmod m,\quad x_i\in\N,\quad x_i\le m^{1/(4e^{1/2})+\varepsilon}.
$$

We shall prove the following result.

\begin{theorem}
\label{thm:Zm}
For any $0<c_0<1$ there exists a positive integer $n=n(c_0)$ and a number $\delta=\delta(c_0)>0$ such that the following holds: let $c_0\le c<1$ and
$$
\cA=\{x\!\!\!\pmod m ;\, \, 1\le x\le  m^{c},\, \gcd(x,m)=1\}.
$$
Then the set  $\cA^{n}$ is a subgroup of the multiplicative group $\Z_m^*$ and
$$
|\cA^{n}|>\delta \phi(m).
$$
\end{theorem}

Here, as usual, $\phi(\cdot)$ is the Euler's totient function, $\Z_m^*$ is the multiplicative group of invertible
classes modulo $m$ and $\cA^n$ is the $n$-fold product set of $\cA$, that is,
$$
\cA^n=\{x_1\cdots x_n;\quad x_i\in \cA\}.
$$
Recall that $|\Z_m^*|=\phi(m)$.

From Theorem~\ref{thm:Zm} we shall derive the following consequences.

\begin{corollary}
\label{cor:Zm-1}
For any $\varepsilon>0$ there exists a positive integer $k=k(\varepsilon)$ such that for any sufficiently large positive integer $m$ the congruence
$$
x_1\cdots x_{k}\equiv 1\!\!\! \pmod m,\quad x_i\in\N,\quad x_i\le m^{\varepsilon}
$$
has a solution with $x_1\not=1$.
\end{corollary}

\begin{corollary}
\label{cor:Zp}
There exists an absolute constant $n_0\in\N$ such that for any $0<\varepsilon<1$ and any sufficiently large prime $p>p_0(\varepsilon)$, every quadratic residue $\lambda$ modulo $p$ can be represented in the form
$$
x_1\cdots x_{n_0}\equiv \lambda\!\!\!\pmod p,\quad x_i\in\N,\quad x_i\le p^{1/(4e^{2/3})+\varepsilon}.
$$
\end{corollary}

\section{Proof of Theorems~\ref{thm:13 int},\ref{thm:12 +1 int}}

The proof of Theorems~\ref{thm:13 int},\ref{thm:12 +1 int} is based on the arguments of Ayyad and Cochrane~\cite{AC} with some modifications.

\begin{lemma}
\label{lem: card XN} Let $N<p$ be a positive integer, $\cX\subset \{1,2,\ldots, p-1\}$. Then for any fixed integer constant $n_0>0$ we have
$$
|\{xy\!\!\!\!\pmod p; \,\, x\in \cX,\, 1\le y\le N\}|> \Delta |\cX|,
$$
where
$$
\Delta =\min\Bigl\{\Bigl(\frac{p}{|\cX|}\Bigr)^{1/n_0},\, \frac{N}{|\cX|^{1/n_0}}\Bigr\} N^{o(1)}
$$
as $N\to\infty$.
\end{lemma}

\begin{proof}

Let $J$ be the number of solutions of the congruence
$$
x_1y_1\equiv x_2y_2\!\!\!\pmod p,\quad x_1,x_2\in \cX, \quad 1\le y_1,y_2\le N.
$$
Then
$$
J=\frac{1}{p-1}\sum_{\chi}\Bigl|\sum_{x\in \cX}\chi(x)\Bigr|^2\Bigl|\sum_{y=1}^{N}\chi(y)\Bigr|^2.
$$
Therefore, by the H\"older inequality we get
\begin{equation}
\label{eqn: L1 J < A B}
J\le A^{(n_0-1)/n_0} B^{1/n_0},
\end{equation}
where
\begin{equation}
\label{eqn: L1 A < X power}
A=\frac{1}{p-1}\sum_{\chi}\Bigl|\sum_{x\in \cX}\chi(x)\Bigr|^{2n_0/(n_0-1)}, \quad B=\frac{1}{p-1}\sum_{\chi}\Bigl|\sum_{y=1}^{N}\chi(y)\Bigr|^{2n_0}.
\end{equation}
Next, we have
$$
A\le |\cX|^{2/(n_0-1)}\Bigl(\frac{1}{p-1}\sum_{\chi}\Bigl|\sum_{x\in \cX}\chi(x)\Bigr|^{2}\Bigr)=|\cX|^{(n_0+1)/(n_0-1)}.
$$
The quantity $B$ is equal to the number of solutions of the congruence
$$
y_1\cdots y_{n_0}\equiv y_{n_0+1}\cdots y_{2n_0}\pmod p,\quad 1\le y_i\le N.
$$
We express the congruence as the equation
$$
y_1\cdots y_{n_0} = y_{n_0+1}\cdots y_{2n_0}+pz,\quad 1\le y_i\le N,\, z\in \Z.
$$
Note than $|z|\le N^{n_0}/p.$ Hence, there are at most
$$
\Bigl(\frac{2N^{n_0}}{p}+1\Bigr)N^{n_0}
$$
possibilities for $(y_{n_0+1},\ldots, y_{2n_0}, z)$. from the estimate for the divisor function it follows that, for each fixed  $y_{n_0+1},\ldots, y_{2n_0}, z$
there are at most $N^{o(1)}$ possibilities for $y_1,\ldots, y_{n_0}$. Therefore,
$$
B\le \Bigl(\frac{N^{n_0}}{p}+1\Bigr)N^{n_0+o(1)}.
$$
Incorporating this and~\eqref{eqn: L1 A < X power} in~\eqref{eqn: L1 J < A B}, we obtain
$$
J\le |\cX|^{(n_0+1)/n_0}\Bigl(\frac{N}{p^{1/n_0}}+1\Bigr)N^{1+o(1)}.
$$
Therefore, from the relationship between the number of solutions of a symmetric congruence and the cardinality of the corresponding set, it follows
\begin{eqnarray*}
|\{xy\!\!\!\pmod p; \,\, x\in \cX,\, 1\le y\le N\}| \qquad \qquad \qquad \qquad \qquad \qquad\\  \qquad \qquad \ge \frac{|\cX|^2 N^2}{J}\ge \min\{|\cX|^{(n_0-1)/n_0}p^{1/n_0}, |\cX|^{(n_0-1)/n_0} N\}N^{o(1)},
\end{eqnarray*}
which concludes the proof of Lemma~\ref{lem: card XN}.
\end{proof}

\begin{lemma}
\label{lem: card XI} Let $\cX\subset \{1,2,\ldots, p-1\}$ and let $\cI\subset \{1,2,\ldots, p-1\}$ be an interval with $|\cI|>p^{1/4+\varepsilon}$, where $\varepsilon>0$. Then
$$
|\{xy\!\!\!\pmod p; \,\, x\in \cX,\, y\in \cI\}|> 0.5\min\{p, |\cX|p^{c}\}
$$
for some $c=c(\varepsilon)>0$.
\end{lemma}

\begin{proof}

As in the proof of Lemma~\ref{lem: card XN}, we let $J$ be the number of solutions of the congruence
$$
x_1y_1\equiv x_2y_2\!\!\!\pmod p,\quad x_1,x_2\in \cX, \quad y_1,y_2\in\cI.
$$
Then
$$
J=\frac{1}{p-1}\sum_{\chi}\Bigl|\sum_{x\in \cX}\chi(x)\Bigr|^2\Bigl|\sum_{y\in \cI}\chi(y)\Bigr|^2.
$$
Since $|\cI|>p^{1/4+\varepsilon}$, from the  well-known character sum estimates of Burgess~\cite{Bur1, Bur2}, we have
$$
\Bigl|\sum_{n\in\cI} \chi(n)\Bigr| < |\cI|p^{-\delta},\quad \delta=\delta(\varepsilon)>0,
$$
for  any non-principal character $\chi\!\!\pmod p$. Therefore, separating the term that corresponds to the principal character $\chi=\chi_0$, we get
$$
J\le \frac{|\cX|^2|\cI|^2}{p-1}+|\cI|^2p^{-2\delta}\Bigl(\frac{1}{p-1}\sum_{\chi}\Bigl|\sum_{x\in \cX}\chi(x)\Bigr|^2\Bigr)=\frac{|\cX|^2|\cI|^2}{p-1}+|\cX||\cI|^2p^{-2\delta}.
$$
Hence,
$$
|\{xy\!\!\!\pmod p; \, x\in \cX,\, y\in \cI\}|\ge \frac{|\cX|^2|\cI|^2}{J}\ge 0.5\min\{p, |\cX|p^{\delta}\}.
$$

\end{proof}

In what follows, the elements of $\F_p$ will be represented by their concrete representatives from the set of integers $\{0,1,\ldots, p-1\}$.

Following the lines of the work of Ayyad and Cochrane~\cite{AC},  we appeal to the result of Hart and Iosevich~\cite{HI}.

\begin{lemma}
\label{lem: HI}
Let $\cA,\cB,\cC,\cD$ be subsets of $\F_p^*$
satisfying
$$
|\cA||\cB||\cC||\cD|>p^{3}.
$$
Then
$$
\F_p^*\subset \cA\cB+\cC\cD.
$$
\end{lemma}

We also need the following consequence of~\cite[Corollary 18]{BGKSh}.

\begin{lemma}
\label{lem: BGKSh}
Let $h<p^{1/4}$ and let $\cA_1,\cA_2,\cA_3\subset \F_p^*$ be intervals of cardinalities $|\cA_i|>h$, i=1,2,3. Then
$$
|\cA_1\cA_2\cA_3|\ge \exp\Bigl(-C\frac{\log h}{\log\log h}\Bigr)h^3.
$$
for some constant $C$.
\end{lemma}

Now we proceed to derive Theorems~\ref{thm:13 int},\ref{thm:12 +1 int}. Let $p^{0.1}<h<p^{1/4}$ to be defined later and assume that
$$
|\cI_{i}|>h,\quad i=1,2,\ldots, 12.
$$
Define
$$
\cX=\cI_7\cI_8\cI_9, \quad \cA=\cI_1\cI_2\cI_3,\quad \cB= \cI_4\cI_5\cI_6, \quad \cC= \cI_7\cI_8\cI_9,\quad \cD=\cX\cI_{13}.
$$
From Lemma~\ref{lem: BGKSh} we have that $|\cX|>h^{3+o(1)}$ and
$$
|\cA||\cB||\cC|> h^{9+o(1)}.
$$
Now we observe that Lemmas~\ref{lem: card XN},\ref{lem: card XI} imply that
\begin{equation}
\label{eqn: D>h3+}
|\cD|=|\cX\cI_{13}|>h^{3+\delta_0}
\end{equation}
for some $\delta_0=\delta_0(\varepsilon)>0.$ Indeed, this is trivial for $|\cX|>h^{3.1}$, so let  $|\cX|<h^{3.1}$. Then
in the case of Theorem~\ref{thm:13 int} the estimate~\eqref{eqn: D>h3+} follows from Lemma~\ref{lem: card XI}.
In the case of Theorem~\ref{thm:12 +1 int} we apply Lemma~\ref{lem: card XN} with $N=\lfloor p^{\varepsilon}\rfloor$ and $n_0=\lceil 1/\varepsilon\rfloor$, and obtain that
$$
|\cD|>|\cX||\cI_3|^{\delta}>h^{3+0.9\delta}
$$
for some $\delta=\delta(\varepsilon)>0.$

Thus, we have~\eqref{eqn: D>h3+}, whence
$$
|\cA||\cB||\cC||\cD|>h^{12+0.9\delta_0}.
$$
Therefore, there exists $c=c(\varepsilon)>0$ such that if $h=p^{\frac{1}{4}-c}$, then we get
$$
|\cA||\cB||\cC||\cD|>p^3.
$$
Theorems~\ref{thm:13 int},\ref{thm:12 +1 int} now follow by appealing to Lemma~\ref{lem: HI}.

\section{Proof of Theorem~\ref{thm:Zm}}

Let $\cG$ be an abelian group written multiplicatively and let $\cX\subset G$. The set $\cX$ is a basis of order $h$ for $\cG$ if $\cX^h=\cG$. This definition implies that if $1\in \cX$ and
$\cX$ is a basis of order $h$ for $\cG$, then $\cX$ is also a basis of order $h_1$ for $\cG$ for any $h_1\ge h.$

We need the following consequence
of a result of Olson~\cite[Theorem 2.2]{Ol} given in Hamidoune and R{\"o}dseth~\cite[Lemma 1]{HR}.

\begin{lemma}
\label{lem: Olson}
Let $\cX$ be a subset of $\cG$. Suppose that $1\in \cX$ and that $\cX$ generates $\cG$. Then $\cX$ is a basis for $\cG$ of order at most
$\max\left\{2, \frac{2|\cG|}{|\cX|}-1\right\}.$
\end{lemma}

We recall that $\Psi(x;y)$ denotes the number of $y$-smooth positive integers $n\le x$ (that is the number of positive integers $n\le x$ with no prime divisors greater than $y$),
and $\Psi_q(x;y)$ denotes the number of $y$-smooth positive integers $n\le x$ with $\gcd(n,q)=1$. It is well known that
for any $\varepsilon>0$ there exists $\delta=\delta(\varepsilon)>0$ such that
$\Psi(m;m^{\varepsilon})\ge \delta m$. We need the following lemma, which follows from~\cite[Theorem 1]{FT}.

\begin{lemma}
\label{lem: TPh}
For any $\varepsilon>0$ there exists $\delta=\delta(\varepsilon)>0$ such that
$$
\Psi_m(m;m^{\varepsilon})>\delta\phi(m).
$$
\end{lemma}

We proceed to prove Theorem~\ref{thm:Zm}. Let $\cS=\cS(c_0, m)$ be the set of $m^{c_0}$-smooth positive integers $n\le m$ with $\gcd(n,m)=1$. As mentioned in~\cite{HSh}, if $x\in \cS$,
then we can combine the prime divisors of $x$ in a greedy way into factors of size at most $m^{c}$.
More precisely, we can write
$x=x_1\cdots x_k$
such that $x_1\le m^{c}$ and  $m^{c/2}\le x_j\le m^{c}$ for $j=2,\ldots, k.$
In particular, we have
$$
(k-1)c_0/2 \le (k-1)c/2\le 1.
$$
Hence, $k\le 2/c_0 +1$, and since $1\!\!\pmod m\in \cA$,  it follows that
$$
\cS\!\!\!\pmod m\subset \cA^{n_1};\quad n_1=\lceil2/c_0\rceil +1.
$$
In particular, by Lemma~\ref{lem: TPh} we have
\begin{equation}
\label{eqn: A n1 card}
|\cA^{n_1}|\ge |\cS|=\Psi_m(m;m^{\varepsilon})>\delta\phi(m)
\end{equation}
for some $\delta=\delta(c_0)>0.$

Let $h$ be the smallest positive integer such that $\cA^{n_1h}$ is a subgroup of $\Z_m^*$. Applying Lemma~\ref{lem: Olson} with $\cG=\cA^{n_1h}$ and $\cX=\cA^{n_1}$, we get that
$$
h\le 1+\frac{2|\cA^{n_1h}|}{|\cA^{n_1}|}\le 1+\frac{2|\Z_m^*|}{|\cA^{n_1}|}\le 1+\frac{2\phi(m)}{\delta\phi(m)} = 1+2\delta^{-1}.
$$
Therefore, since  $1\!\!\pmod m\in \cA$, we get that for $n=(1+\lceil 2\delta^{-1}\rceil)n_1$ the set $\cA^{n}$ is a multiplicative subgroup of $\Z_m^*$. Taking into account~\eqref{eqn: A n1 card},
we conclude the proof of Theorem~\ref{thm:Zm}.

Let now $g$ be any element of the group $\cA^{n}$ distinct from $1\pmod m$. We also have that $g^{-1}\in \cA^{n}.$ Thus, Corollary~\ref{cor:Zm-1}, with $k=2n$, follows from the representation
$gg^{-1}=1\!\!\pmod m.$

We shall now prove Corollary~\ref{cor:Zp}. Let
$$
\cA=\{x\!\!\!\pmod p;\,\, x\in\N,\, x\le p^{1/(4e^{2/3})+\varepsilon}\}.
$$
In Theorem~\ref{thm:Zm}, we take $m=p$,  $c_0=1/(4e^{2/3})$ and $c=1/(4e^{2/3})+\varepsilon$. Thus, there is an absolute constant $n_0$ such that $\cA^{n_0}$ is a subgroup of $\F_p^*$
and $|\cA^{n_0}|>\delta_0 (p-1)$ for some absolute constant $\delta_0>0$. In other words, there is an integer $\ell | p-1$ with $1\le \ell\le 1/\delta_0$ such that
$$
\cA^{n_0}=\{x^{\ell}\!\!\!\pmod p;\,\, 1\le x\le p-1\}.
$$
Let $t=t(\ell,p)$ be the smallest positive  $\ell$-th power nonresidue modulo $p$. According to the well-known consequence  of Vinogradov's work~\cite{Vin} combined with the Burgess character sum estimate~\cite{Bur1, Bur2}, we have that
$$
t\le p^{1/(4e^{(\ell-1)/\ell})+\varepsilon/2}.
$$
On the other hand, since $t\not\in \cA^{n_0}$ we have $t\ge p^{1/(4e^{2/3})+\varepsilon}$. Hence, $\ell\in\{1,2\}$ and the claim follows.

M.~Z.~Garaev, Centro de Ciencias Matem\'{a}ticas, Universidad
Nacional Aut\'onoma de M\'{e}xico, C.P. 58089, Morelia,
Michoac\'{a}n, M\'{e}xico.

Email: {\tt garaev@matmor.unam.mx}

\end{document}